\documentclass[a4paper,10pt]{article}
 \usepackage{benstyle}
\newtheorem{algorithm}[theorem]{Algorithm}

\title{Complexity Issues in Computing Spectra, Pseudospectra and Resolvents}
\author{Anders C. Hansen \\ DAMTP, Centre for Mathematical Sciences \\ University of Cambridge \\ Wilberforce Rd, Cambridge CB3 0WA \\ United Kingdom\\ \hspace{20pc} \and 
Olavi Nevanlinna \\Department  of Mathematics and Systems Analysis \\ Aalto University\\ FI-00076 Aalto\\ Finland}
\begin{document}
\date{}
\maketitle

\begin{abstract}
We display methods that allow for computations of spectra,
pseudospectra and resolvents of linear operators on Hilbert spaces
and also elements in unital Banach algebras. The paper considers two
different approaches, namely, pseudospectral techniques and polynomial
numerical hull theory. The former is used for Hilbert space operators
whereas the latter can handle the general case of elements in a Banach
algebra. This approach leads to multicentric holomorphic calculus.
We also discuss some new types of pseudospectra and the recently defined
Solvability Complexity Index.   

\end{abstract}


\section{Introduction}
 
The theme of this article is how to compute and approximate
spectra of arbitrary closed operators on separable Hilbert
spaces and arbitrary elements of a unital Banach algebra. 
This task is a fascinating
mathematical problem, but it is strongly motivated by
applications. The reason is as follows. After the triumph of quantum
mechanics, operator and spectral theory became indispensable
mathematical disciplines in order to support quantum theory and also to
secure its mathematical foundations. The fundamental paper of Murray and von Neumann \cite{Von_Neumann1}
(which was also very much motivated by quantum mechanics) on operator algebras showed how 
important abstractions of these issues were. In particular, one could then ask spectral questions not only 
about operators but also about elements in a unital $C^*$-algebra (or
a von Neumann algebra), or more generally, a unital Banach algebra. 
There is a vast literature on how
to analyze spectra of linear operators and the field is still very
much active. 

So far, so good, the only problem is that the theoretical
 physicist may not only want theorems about structural properties of
 spectra, one may actually want to determine the spectra
 completely. When faced with this problem the mathematician may first
 recall that even if the dimension of the Hilbert space is finite, this
 is not trivial. One quickly realizes that, due to Abel's contribution
 on the unsolvability of the quintic using radicals, one is doomed to
 fail if one tries to construct the spectrum in terms of finitely many
 arithmetic operations and radicals of the matrix elements of the
 operator. However, in finite dimensions, there is a vast theory on
 how to obtain sequences of sets, whose construction only require
 finitely many arithmetic operations and radicals of the matrix
 elements, 
such that the
 sequence converges to the spectrum of the desired operator. 
Thus, at least in finite dimensions, one can construct the spectrum, 
and this construction automatically yields a
method for approximating the spectrum. Even though this may be
difficult in practice, one has a mathematical theory that guarantees
that up to an arbitrarily small error, one can determine the spectra of
operators on finite dimensional Hilbert spaces.    

There is no automatic extension from the finite dimensional case, 
and the problem is therefore: what can be done in infinite dimensions?
Moreover, how does one handle the case of an unbounded operator? Keeping the
Schr\"{o}dinger and Dirac operators in mind, one realizes that the
unbounded cases may be the most important ones. We must emphasize that
quite a lot is known about how to approximate spectra of
Schr\"{o}dinger and Dirac operators, but, as far as we know, even in
the self-adjoint case, one still only knows how to deal with special
cases, and current methods lack generality.  

Bill Arveson commented in 1994  \cite{Arveson_cnum_lin94} on the 
situation of the computational spectral problem:
``Unfortunately, there is a dearth of literature on this basic
problem, and so far as we have been able to tell, there are no proven
techniques.'' 
We will emphasize that this quote is concerned with the general
problem, and if one has more structure available
e.g. self-adjointness, then much more can be said. However, during the
last two decades the importance of non-normal operators and their
spectra has become increasingly evident. In particular,
the growing
interest in non-Hermitian quantum mechanics
\cite{Bender3,Hatano_96},
non-self-adjoint differential operators \cite{Davies_bull,Zworski} and
in general non-normal phenomena
\cite{Trefethen_paper,Trefethen_Embree} 
has made non-self-adjoint
operators and pseudospectral theory indispensable. This emphasizes the
importance of the general problem and poses a slightly
philosophical problem, namely, could there be operators whose spectra
we can never determine? If such operators are indispensable
in areas of mathematical physics it may lead to serious restrictions to
our possible understanding of some physical systems. 
Fortunately, there have been some recent developments on the topic and due to 
the results in \cite{Hansen2} the future may not be so pessimistic.
For other papers related to the ideas presented in this papers we refer to 
\cite{Arveson_cnum_lin94, Arveson_role_of94,Bedos_folner97, Bottcher_Pseudo, Bottcher_C*, 
Bottcher_99,Lyonell_2, Lyonell1,Marletta2,Brown,Iserles, Silbermann, Hansen3,Levitin}.

\subsection{Background and Notation}
We will in this section review some basic definition and introduce the
notation used in the article. Throughout the paper $\mathcal{H}$ will
always denote a separable Hilbert space, $\mathcal{X}$ and  $\mathcal{Y}$ are 
Banach spaces, $\mathcal{B}(\mathcal{H})$ and $\mathcal{B}(\mathcal{X})$ the
sets of bounded linear operators on $\mathcal{H}$ and $\mathcal{X}$, 
$\mathcal{C}(\mathcal{H})$ and $\mathcal{C}(\mathcal{X})$ the sets
of densely defined closed linear operators on $\mathcal{H}$ and 
$\mathcal{X}$, and 
$\mathcal{SA}(\mathcal{H})$ the set of self-adjoint operators on 
$\mathcal{H}.$ For $T
\in \mathcal{C}(\mathcal{H})$ or $\mathcal{C}(\mathcal{X})$ the domain
of $T$ will be denoted by
$\mathcal{D}(T).$ 
Furthermore, $\mathcal{A}$ denotes a complex unital Banach algebra.  We denote by $a,b,\dots $ generic elements in the algebra and the unit $e \in \mathcal{A}$ satisfies $\|e\| = 1$.
Thus for example, the spectrum of $a\in \mathcal{A}$ is  given by
\begin{equation*}\label{perusmaaritelma}
\sigma(a)= \{ z\in \mathbb C \ : \  z-a   \text { does not have an inverse} \}.
\end{equation*}
If $\mathcal X$ is a Banach space (or Hilbert space as before) then  bounded operators   $B(\mathcal X)$ is an important example of  unital Banach algebras.  The spectrum of an operator $T \in \mathcal{C}(\mathcal{X})$ is defined slightly differently from the Banach algebra case:
\begin{equation*}\label{rajoperspektri}
\sigma(T) =\{ z\in \mathbb C \  : \  z-T   \text { does not have an inverse in } B(\mathcal X)  \}.
\end{equation*}
We will denote orthonormal basis elements of 
$\mathcal{H}$
by $e_j$, and if $\{e_j\}_{j \in \mathbb{N}}$ is a basis 
and $\xi \in \mathcal{H}$
then $\xi_j = \langle \xi, e_j \rangle.$ The word basis will always
refer to an orthonormal basis. If $\mathcal{H}$ is a finite
dimensional Hilbert space with a basis $\{e_j\}$ then
$LT_{\mathrm{pos}}(\mathcal{H})$ will denote the set of lower
triangular matrices (with respect to $\{e_j\}$) 
with positive elements on the diagonal. The closure of a set $\Omega
\in \mathbb{C}$ will be denoted by $\overline{\Omega}$ or
$\mathrm{cl}(\Omega),$  and the interior of $\Omega$ will be denoted by $\Omega^o$.

Convergence of
sets in the complex plane will be quite crucial in our analysis and
hence we need the Hausdorff metric as defined by the following. 
\begin{definition}
For a set $\Sigma \subset \mathbb{C}$ and $\delta > 0$ we
  will let $\omega_{\delta}(\Sigma)$ denote the $\delta$-neighborhood of
  $\Sigma$ (i.e. the union of all $\delta$-balls centered at points of
  $\Sigma).$
Given two compact sets $\Sigma, \Lambda \subset \mathbb{C}$ their
Hausdorff distance is 
\[
d_H(\Sigma,\Lambda) = \max \left\{\sup_{\lambda \in \Sigma}
d(\lambda,\Lambda), \sup_{\lambda \in \Lambda} d(\lambda,\Sigma)\right\}
\]  
where $d(\lambda,\Lambda) = \inf_{\rho \in \Lambda}|\rho-\lambda|.$
If $\{\Lambda_n\}_{n \in \mathbb{N}}$ is a sequence of compact subsets
of $\mathbb{C}$ and $\Lambda \subset \mathbb{C}$ is compact such that
$d_H(\Lambda_n, \Lambda) \rightarrow 0$ as $n \rightarrow \infty$  we
may use the notation
$
\Lambda_n \longrightarrow \Lambda.
$
\end{definition}

When it comes to unbounded subsets of $\mathbb{C}$ one has to be a little careful, as the Hausdorff metric is no longer applicable. Instead one may use the \emph{Attouch-Wets} metric defined by 
\begin{equation*}
	d_{\mathrm{AW}}(\Sigma,\Lambda)=\sum_{i=1}^\infty2^{-i}\min\left\{1,\sup_{|x|<i}\left|d(x,\Sigma)-d(x,\Lambda)\right|\right\},
	\end{equation*}
where $\Sigma$ and $\Lambda$ are closed subsets of $\mathbb{C}$, and where $d(x,\Sigma)$ is as above, which is well-defined even when $\Sigma$ is unbounded. Note that the Attouch-Wets metric becomes, in some sense, a metric that represents {\it locally uniform} convergence of sets. This can easily be seen as follows. Let $\Sigma \subset\mathbb{C}$ and $\Sigma_n\subset\mathbb{C},\ n=1,2,\dots$ be closed and non-empty. Then
	\begin{equation*}
	d_{\mathrm{AW}}(\Sigma_n,\Sigma)\to0\quad\text{if and only if}\quad d_\mathcal{K}(\Sigma_n,\Sigma)\to0\text{ for any compact } \mathcal{K} \subset\mathbb{C},
	\end{equation*}
where
	\begin{equation*}\label{compact_conv}
	d_\mathcal{K}(\Sigma,\Lambda)=\max\left\{\sup_{s\in \Sigma\cap \mathcal{K}}d(s,\Lambda),\sup_{t\in \Lambda \cap \mathcal{K}}d(t,\Sigma)\right\},
	\end{equation*}
	where we use the convention that $\sup_{s\in \Sigma\cap \mathcal{K}}d(s,\Lambda) = 0$ if $ \Sigma\cap \mathcal{K} = \emptyset$. This also makes it clear that if we deal with bounded sets, the two metrics are equivalent.

\subsection{Overview of the paper} 

The discussion in this paper touches two different approaches for computing the spectrum. In {\it pseudospectral techniques} one studies the sets in which the resolvent of the operator $T$,
$$
z\mapsto (z-T)^{-1},
$$
becomes large (or does not  exist).  {\it Polynomial numerical hulls} on the other hand are based on sets
$$
V_p(T) =\{z \in \mathbb C \ : \ |p(z)| \le \|p(T)\| \}
$$
where $p$ is  a monic polynomial, and thus they do not have recourse to any inversion of  operators.

We start with the pseudospectral theory and present some important variants to deal with operators.  Then we define the {\it Solvability Complexity Index} which keeps track on how many  levels an algorithm has limiting processes.  We then discuss compact operators and show that  the spectrum can be computed  using  algorithms with index 1.  The next sections concern the indices of  bounded and unbounded operators respectively.
 
Then we turn to polynomial numerical hull techniques. What is crucial  is that all sets considered are {\it inclusion sets} for the spectrum. For example, intersecting the $V_p(T)$ over all first degree polynomials gives the closure of the numerical range while intersecting over all polynomials equals the polynomially convex hull of the spectrum.
The algorithms have a natural set up within Banach algebras and for example, the spectrum of an algebra element (within the subalgebra it generates) can be computed with just one limiting process {\it provided} that one assumes that the norm of an element is available as a single operation.

It turns out that, as a byproduct, one obtains an explicit representation for the resolvent which then leads to an approach for algorithmic  holomorphic calculus with low complexity.

\section{Pseudospectra and their Close Cousins}\label{morespseudo}

Pseudospectral theory is now a mainstay in spectral analysis and we
refer the reader to \cite{Trefethen_Embree} for a thorough introduction to the
topic. Before we introduce other types of pseudospectra  
we will recall the definition of the pseudospectrum and discuss some
of its quite pleasant properties.

\begin{definition}
Let $T$ be a closed operator on a Hilbert space $\mathcal{H}$ such
that $\sigma(T) \neq \mathbb{C},$ and $\epsilon > 0.$  
The $\epsilon$-pseudospectrum of $T$ is
defined as the set
$$
\sigma_{\epsilon}(T) = \sigma(T) \cup \{z \notin \sigma(T):
\|(z - T)^{-1}\| >
\epsilon^{-1}\}.
$$
\end{definition}  

Note that, for $\epsilon > 0$, the mapping 
$\mathcal{B}(\mathcal{H}) \ni T \mapsto
\overline{\sigma_{\epsilon}(T)} \in \mathcal{F}$, where $\mathcal{F}$ denotes
the collection of compact subsets of $\mathbb{C}$ equipped with the
Hausdorff metric, is a continuous mapping. This is a nice property
that is not shared by the spectrum. In particular, it is well known
that the mapping  
$\mathcal{B}(\mathcal{H}) \ni T \mapsto
\sigma(T) \in \mathcal{F}$ is discontinuous (see \cite{Hansen2}
for examples). The nice property of continuity is one reason why the 
pseudospectrum has become popular and very important in applications
\cite{Hansen3}, although there are many other justifications for the
enthusiasm for the pseudospectum \cite{Trefethen_Embree, Trefethen_Acta}. 
Note, however, 
that the continuity property ends with $\mathcal{B}(\mathcal{H})$. In
particular, as a result of the fundamental paper by Shargorodsky
\cite{Sharg2}, the mapping  
$\mathcal{C}(\mathcal{H}) \ni T \mapsto
\overline{\sigma_{\epsilon}(T)},$ where 
$\mathcal{C}(\mathcal{H})$ is equipped with the graph metric (see
\cite{Hansen2}), may not be continuous. This is also the case in the
bounded example if the Hilbert space is replaced by a Banach space. 

In \cite{Hansen1} and \cite{Hansen2} some new variants of the
pseudospectrum were introduced. We will recall these sets here and
discuss some of their nice properties. 
Before we introduce the different pseudospectra we need to define a
 convenient function. 
 
 \begin{definition}\label{functions}
 Define, for $n \in \mathbb{Z_+}$ the function 
 $\Phi_n: \mathcal{B}(\mathcal{H}) \times \mathbb{C}
\rightarrow \mathbb{R}$ by
$$
\Phi_{n}(S,z) = \min\left\{\lambda^{1/2^{n+1}}: 
\lambda \in  \sigma \left(((S - z)^*)^{2^n} (S -
z)^{2^n}\right)\right\}.
$$     
Define also for $T \in \mathcal{B}(\mathcal{H})$
\begin{equation}\label{gamma}
\gamma_n(z)  = \min[\Phi_n(T,z),\Phi_n(T^*,\bar z)].
\end{equation}
\end{definition}

\subsection{The $(n,\epsilon)$-pseudospectrum}
 
We will now introduce the $(n,\epsilon)$-pseudospectrum, which is an extension of the usual pseudospectrum.  
(This set was actually first introduced in \cite{Hansen1}.)
\begin{definition}
Let $T$ be a closed operator on a Hilbert space $\mathcal{H}$ such
that $\sigma(T) \neq \mathbb{C},$ and 
let $n \in \mathbb{Z}_+$ and $\epsilon > 0.$  
The $(n,\epsilon)$-pseudospectrum of $T$ is
defined as the set
$$
\sigma_{n,\epsilon}(T) = \sigma(T) \cup \{z \notin \sigma(T):
\|(z - T)^{-2^n}\|^{1/2^n} >
\epsilon^{-1}\}.
$$
\end{definition}  
As we will see in the next theorem,
the $(n,\epsilon)$-pseudospectrum has all the nice
continuity properties that the pseudospectrum has, but it also
approximates the spectrum arbitrarily well for large $n.$
\begin{theorem}\label{bounded} \cite{Hansen2} 
Let
$T \in \mathcal{B}(\mathcal{H})$, $\gamma_n$ be defined as in
(\ref{gamma}) and $\epsilon > 0$.
Then the following is true:
\begin{itemize}
\item[(i)]
 $\sigma_{n+1,\epsilon}(T) \subset \sigma_{n,\epsilon}(T).$
\item[(ii)]
$\sigma_{n,\epsilon}(T) = \{z \in \mathbb{C}: \gamma_n(z) <
  \epsilon\}.$ 
\item[(iii)]
$
\mathrm{cl}(\{z : \gamma_n(z) <
  \epsilon\}) = \{z : \gamma_n(z) \leq
  \epsilon\}.
$
\item[(iv)]
Let $\omega_{\epsilon}(\sigma(T))$ denote the $\epsilon$-neighborhood
around 
$\sigma(T).$ Then
$$
d_H\left(\overline{\sigma_{n,\epsilon}(T)},
  \overline{\omega_{\epsilon}(\sigma(T))}\right) \longrightarrow 0, \qquad n
  \rightarrow \infty. 
$$ 
\item[(v)] 
If $\{T_k\} \subset \mathcal{B}(\mathcal{H})$ and $T_k \rightarrow T$
in norm, it follows that 
$$d_H\left(\overline{\sigma_{n,\epsilon}(T_k)},
  \overline{\sigma_{n,\epsilon}(T)}\right) \longrightarrow 0, \qquad k \rightarrow \infty.$$
\end{itemize}
\end{theorem}

Note that the definition of the $(n,\epsilon)$-pseudospectrum carries 
over to the abstract case of a unital 
Banach algebra. In particular, we will have the following definition.
\begin{definition}
Let $a \in \mathcal{A}$ where $ \mathcal{A}$ is a unital Banach
algebra, 
$n \in \mathbb{Z}_+$ and $\epsilon > 0.$  
The $(n,\epsilon)$-pseudospectrum of $a$ is
defined as the set
$$
\sigma_{n,\epsilon}(a) = \sigma(a) \cup \{z \notin \sigma(a):
\|(z-a)^{-2^n}\|^{1/2^n} >
\epsilon^{-1}\}.
$$
\end{definition}  
However, all of the properties (i)-(v) may not carry over in the
abstract case. 
For example, (ii) and (iii) does not 
even make sense, since an involution mapping $a \mapsto a^*$ is needed
to define the function $\gamma$.  Also, property  
(v) may not carry over. In fact, by the fundamental result of
Shargorodsky \cite{Sharg2}, it is known that pseudospectra of Banach space operators
can "jump". More precisely one can have that  
$$
\overline{\sigma_{\epsilon}(a)}  \neq \sigma(a) \cup \{z \notin \sigma(a):
\|(z-a)^{-1}\|  \geq
\epsilon^{-1}\}
$$ 
for $a$ being a bounded operator on a certain Banach space. 
However,  properties (i) and (iv) carry over and this can be seen as
follows. First, we need the following  
proposition from \cite{Hansen2} : 
\begin{proposition}\label{boundary} \cite{Hansen2} 
Let $f : \mathbb{C} \rightarrow [0,\infty)$ be continuous 
and let $\{f_k\}_{k \in \mathbb{N}}$ be a sequence of functions
such that $f_k :
  \mathbb{C} \rightarrow [0,\infty)$ and $f_k \rightarrow f$
  locally uniformly. Suppose that one of the two following properties
  are satisfied.
\begin{itemize}
\item[(i)]
$f_k \rightarrow f$ monotonically from above. 
\item[(ii)]
For $\epsilon > 0$, then 
$
\mathrm{cl}(\{z : f(z) <
  \epsilon\}) = \{z : f(z) \leq
  \epsilon\}.
$
\end{itemize}
Then, it follows that 
$$
d_{\mathrm{AW}}(\mathrm{cl}(\{z : f_k(z) <
  \epsilon\}), \mathrm{cl}(\{z : f(z) <
  \epsilon\})) \longrightarrow 0, \qquad k \rightarrow \infty. 
$$
\end{proposition}
Note that (i) follows by the definition of $\sigma_{n,\epsilon}(a)$ and the
fact that 
\begin{equation*}
\begin{split}
 \|(z-a)^{-2^{n+1}}\|^{-1/2^{n+1}} &\geq
  \|(z-a)^{-2^n}\|^{-1/2^{n+1}} \|(z-a)^{-2^n}\|^{-1/2^{n+1}} \\&=
 \|(z-a)^{-2^n}\|^{-1/2^n}. 
\end{split}
\end{equation*}
To see that (iv) is also true in the general case define 
$$
f(z) = \mathrm{dist}(z,\sigma(a)), 
$$
and define 
$$
f_n(z) = 
\begin{cases} \|(z-a)^{-2^n}\|^{-1/2^n} & z \notin \sigma(a)
\\
0 & z \in \sigma(a).
\end{cases}
$$ 
After observing that 
$$
f(z) =
 1/\rho((z-a)^{-1}),
$$ 
where $\rho((z-a)^{-1})$ denotes the spectral radius of 
$((z-a)^{-1})$, we can apply the spectral radius formula, a
straightforward argument via Dini's Theorem and deduce that 
$f_n \rightarrow f$ locally uniformly and monotonically from above.
By invoking Proposition \ref{boundary} we are done.

\subsection{The Residual Pseudospectrum}

The disadvantage of the $(n,\epsilon)$-pseudospectrum is that even though one
can estimate the spectrum by taking $n$ very large or $\epsilon$
  very small, $n$ may have to be
too large or $\epsilon$ too small for practical purposes (this
would depend on the computer used). Thus, since we only have the
estimate $\sigma(T) \subset \sigma_{n,\epsilon}(T)$ for 
$T \in \mathcal{B}(\mathcal{H}), \epsilon > 0$,
 it is important to get a
``lower'' bound on $\sigma(T)$ i.e. we want to find a set $\theta
\subset \mathbb{C}$ such that $\theta \subset \sigma(T).$ A candidate
for this is described in the following.

\begin{definition}\label{what}
Let $T \in \mathcal{B}(\mathcal{H})$ and $\Phi_0$ be defined as in
Definition 
\ref{functions}. Let   
$\zeta
_1(z) = \Phi_0(T,z)$ 
$\zeta_2(z) =  \Phi_0(T^*,\bar z).$
Now let $\epsilon > 0$ and 
define
the $\epsilon$-residual pseudospectrum to be the set
$$
\sigma_{\mathrm{res},\epsilon}(T) = \{z: \zeta_1(z) > \epsilon, \,
\zeta_2(z) = 0\}
$$
and the adjoint $\epsilon$-residual pseudospectrum to be the set
$$
\sigma_{\mathrm{res^*},\epsilon}(T) = \{z: \zeta_1(z) = 0,\, 
\zeta_2(z) > \epsilon\}.
$$
\end{definition}

\begin{theorem}\label{subharmonic} \cite{Hansen2} 
Let $T \in \mathcal{B}(\mathcal{H})$ and let
$\{T_k\} \subset \mathcal{B}(\mathcal{H})$ such
that $T_k \rightarrow  T$ in norm, as $k \rightarrow \infty.$ Then for
$\epsilon > 0$ we have the following:
\begin{itemize}
\item[(i)]
$\sigma(T) \supset \bigcup_{\epsilon >
  0} \sigma_{\mathrm{res},\epsilon}(T) \cup
  \sigma_{\mathrm{res^*},\epsilon}(T).$ 
\item[(ii)]
$\mathrm{cl}(\{z \in \mathbb{C}: \zeta_1(z) < \epsilon\}) = \{z \in
  \mathbb{C}: \zeta_1(z) \leq 
  \epsilon\}.$ 
\item[(iii)]
$\mathrm{cl}(\{z \in \mathbb{C}: 
\zeta_2(z) < \epsilon \}) = \{z \in \mathbb{C}: 
\zeta_2(z) \leq \epsilon\}.$
\item[(iv)] 
$$
d_H(\mathrm{cl}(\sigma_{\mathrm{res},\epsilon}(T_k)) , 
\mathrm{cl}(\sigma_{\mathrm{res},\epsilon}(T))) \longrightarrow 0, 
\qquad k \rightarrow \infty. 
$$
\item[(v)]
 $$
d_H(\mathrm{cl}(\sigma_{\mathrm{res}^*,\epsilon}(T_k)) , 
\mathrm{cl}(\sigma_{\mathrm{res}^*,\epsilon}(T))) \longrightarrow 0, 
\qquad k \rightarrow \infty.  
$$
\end{itemize}
\end{theorem}

Note that the definition of the residual pseudospectrum and the 
adjoint residual pseudospectrum does not extend to an arbitrary unital
Banach algebra because of the lack of an involution. However, if we
have a unital Banach algebra $\mathcal{A}$ with an involution we can
define the following functions:
Let $a \in \mathcal{A}$ and define 
$$
\zeta_{a,1}(z) = \min\left\{\sqrt{\lambda}: 
\lambda \in  \sigma \left((a - z)^*(a -
z)\right)\right\}
$$
$$
\zeta_{a,2}(z) = \min\left\{\sqrt{\lambda}: 
\lambda \in  \sigma \left((a - z)(a -
z)^*\right)\right\},
$$
then, clearly, Definition \ref{what} extends to elements in
$\mathcal{A}$.
The question on whether properties (i)-(v) in Theorem
\ref{subharmonic} extends to $\mathcal{A}$ is delicate. We will not
make any predictions, however, we would like to point out that 
the proof of Theorem \ref{subharmonic}  \cite{Hansen2}  makes use of a
crucial result, namely   
\begin{theorem}\label{Sharg}(Shargorodsky)\cite{Sharg2}
Let $\Omega$ be an open subset of $\mathbb{C},$ $X$ be a Banach space
and $Y$ be a uniformly convex Banach space. Suppose $A:\Omega
\rightarrow \mathcal{B}(X,Y)$ is an analytic operator valued function
such that $A^{\prime}(z)$ is invertible for all $z \in \Omega.$ If
$\|A(z)\| \leq M$ for all $z \in \Omega$ then $\|A(z)\| < M$ for all
$z \in \Omega.$  
\end{theorem}

\section{The Solvability Complexity Index}\label{index}
Let $T \in \mathcal{B}(\mathcal{H})$ and suppose that we would like to
compute $\sigma(T).$ This is a non-trivial computational problem even
if the dimension of the Hilbert space is finite, say $N$. In that case the
QR-algorithm may be the method of choice and with suitable assumptions
on $T$ one can guarantee that if 
$\theta_n = \{\omega_{1,n}, \hdots, \omega_{N,n}\}$ are
the elements on the diagonal of $T_n,$ where $T_n$ is the $n$-th term
in the QR-iteration, then
$$
d_H(\theta_n, \sigma(T)) \longrightarrow 0, \qquad n \rightarrow \infty. 
$$ 
What is crucial is that there is only one numerical parameter tending
to infinity. Now, suppose that $T \in \mathcal{B}(\mathcal{H})$ is
compact. Let $\{e_j\}$ be a basis for $\mathcal{H}$ and let $P_m$ be
the projection onto $\mathrm{span}\{e_1,\hdots,e_m\}$. In that case it
follows that 
$$ 
d_H(\sigma(P_mT\lceil_{P_m\mathcal{H}}), \sigma(T)) \longrightarrow 0, \qquad
m \rightarrow \infty.   
$$
Thus, a way of computing $\sigma(T)$ is to use the QR-algorithm (or
any other appropriate convergent method, also, for the sake of the argument we
assume that the QR-algorithm converges) on $P_mT\lceil_{P_m\mathcal{H}}$ for
sufficiently large $m.$ Let $\{\omega^{(m)}_{1,n}, \hdots
\omega^{(m)}_{m,n}\}$ denote the elements of the diagonal of 
$\widetilde T_{n,m}$ where $\widetilde T_{n,m}$ is the $n$-th term in
the QR-iteration of the matrix $P_mT\lceil_{P_m\mathcal{H}}$ with
respect to the basis $\{e_j\}.$ In this case 
we may express $\sigma(T)$ as 
$$
\sigma(T) = \lim_{m \rightarrow \infty} \lim_{n \rightarrow \infty} 
\{\omega^{(m)}_{1,n}, \hdots
\omega^{(m)}_{m,n}\}.
$$
Hence, we have a way of computing $\sigma(T)$ but it requires two
limits with indices $n$ and $m$.  The question is then: what about the
general problem? Suppose that we have no more information about the
operator except that it is bounded. How many limits will we need to
compute the spectrum?   
This is the motivation for the Solvability Complexity Index, namely, we want the
Solvability Complexity Index to indicate how many limiting processes do we need to
compute the spectrum (or any other set e.g. the pseudospectrum). 

We will now define the Solvability Complexity Index, but before that we need the
definition of a set of estimating
functions of order $k$. The motivation behind the
definition is that the estimating
functions should be a method used in
actual computations and the integer $k$ determines how many limiting
processes (a la the example above) there are in order to assure
convergence. Note that we also want to include unbounded operators and
thus more formally we have the following.

\begin{definition}\label{Estimating_funct}
Let $\mathcal{H}$ be a Hilbert space spanned by $\{e_j\}_{j \in
  \mathbb{N}}$ and let 
\begin{equation}\label{delta}
\Upsilon = \{T \in \mathcal{C}(\mathcal{H}):
  \mathrm{span}\{e_j\}_{n \in \mathbb{N}} \subset \mathcal{D}(T)\}.
\end{equation}
Let $\mathcal{E} \subset \Upsilon$ and $\Xi :
\mathcal{E} \rightarrow \mathcal{F},$ 
where $\mathcal{F}$ denotes
the family of closed subsets of $\mathbb{C}.$ Let 
$$
\Pi_{\mathcal{E}} = \{\{x_{ij}\}_{i,j \in \mathbb{N}}:\exists\, T \in
\mathcal{E}, \, x_{ij} = \langle Te_j,e_i \rangle \}.
$$
A set of estimating
functions of order $k$ for $\Xi$ is a family of functions 
\begin{equation*}
\begin{split}
\Gamma_{n_1}&:\Pi_{\mathcal{E}}
\rightarrow \mathcal{F},\\
\Gamma_{n_1, n_2}&:\Pi_{\mathcal{E}}
\rightarrow \mathcal{F}, \\
\vdots\\
\Gamma_{n_1, \hdots, n_{k-1}} &:\Pi_{\mathcal{E}} \rightarrow \mathcal{F},
\end{split}
\end{equation*}
$$
\Gamma_{n_1, \hdots, n_k} :
\left\{\{x_{ij}\}_{i,j \leq N(n_1, \hdots ,n_k)}: 
\{x_{ij}\}_{i,j \in \mathbb{N}} \in \Pi_{\mathcal{E}}\right\} 
\rightarrow \mathcal{F},
$$
where $N(n_1, \hdots, n_k) < \infty$ depends on $n_1, \hdots, n_k$,
with the following properties:
\begin{itemize}
\item[(i)]
The evaluation of 
$
\Gamma_{n_1, \hdots, n_k}(\{x_{ij}\})
$
requires only finitely many arithmetic operations and radicals of the
elements $\{x_{ij}\}_{i,j \leq N(n_1, \hdots, n_k)}.$ 
\item[(ii)]
Also, 
we have the following relations between the limits:  
\end{itemize}
\begin{equation*}
\begin{split}
\Xi(T) &= \lim_{n_1 \rightarrow \infty} \Gamma_{n_1}(\{x_{ij}\}), \\
\Gamma_{n_1}(\{x_{ij}\}) &=
  \lim_{n_2 \rightarrow \infty} \Gamma_{n_1, n_2}(\{x_{ij}\}),\\
& \, \, \, \, \vdots\\
\Gamma_{n_1, \hdots, n_{k-1}}(\{x_{ij}\}) &=
  \lim_{n_k \rightarrow \infty} \Gamma_{n_1, \hdots, n_k}(\{x_{ij}\}).
\end{split}
\end{equation*}
The limit is understood to be in the Attouch-Wets metric.  
\end{definition}

\begin{definition}\label{complex_ind}
Let $\mathcal{H}$ be a Hilbert space spanned by $\{e_j\}_{j \in
  \mathbb{N}}$, define $\Upsilon$ as in (\ref{delta}),
and let 
$
\mathcal{E} \subset \Upsilon.
$
A set valued function 
$$
\Xi :
\mathcal{E} \subset \mathcal{C}(\mathcal{H}) \rightarrow \mathcal{F}
$$ is said to
have Solvability Complexity Index $k$ if $k$ is the smallest integer for which
there exists a set of 
estimating functions of order $k$ for $\Xi.$ Also,
$\Xi$ is said to have infinite Solvability Complexity Index if no set of
estimating functions exists. If there is a function
$$
\Gamma : \{\{x_{ij}\}: \exists\, T \in
\mathcal{E}, \, x_{ij} = \langle Te_j,e_i \rangle\}
\rightarrow \mathcal{F}
$$
such that $\Gamma(\{x_{ij}\}) = \Xi(T),$
and the evaluation of $\Gamma(\{x_{ij}\})$ requires
only finitely many arithmetic operations and radicals of a finite
subset of 
$\{x_{ij}\},$ then $\Xi$ is said to have Solvability Complexity Index zero.  
The Solvability Complexity Index of a function $\Xi$ will be denoted
by 
$\mathrm{SCI}(\Xi)$.
\end{definition}

\begin{example}\label{motivation}
\begin{itemize}
\item[(i)]
Let $\mathcal{H}$ be a Hilbert space with basis $\{e_j\},$ 
$\mathcal{E} =
\mathcal{B}(\mathcal{H})$ and $\Xi(T) = \sigma(T)$ for $T \in
\mathcal{B}(\mathcal{H}).$
Suppose that $\mathrm{dim}(\mathcal{H}) \leq 4.$
 Then $\Xi$ must have complexity index zero,
since one can obviously express the eigenvalues of $T$ using finitely
many arithmetic operations and radicals of the matrix elements $x_{ij}
=
\langle Te_j,e_i\rangle.$ 

\item[(ii)]
If $\mathrm{dim}(\mathcal{H}) \geq 5$ then obviously
$\mathrm{SCI}(\Xi) > 0,$ by the much celebrated theory of Abel
on the insolubility of the quintic using radicals. 

\item[(iii)]
Now, what about
compact operators? Suppose for a moment that we can show that  
$\mathrm{SCI}(\Xi) = 1$ if 
$\mathrm{dim}(\mathcal{H}) < \infty.$ A standard way of
determining the spectrum of a compact operator $T$ is to 
let $P_n$ be the
projection onto $\mathrm{span}\{e_j\}_{j \leq n}$ and compute the
spectrum of $P_nA\lceil_{P_n\mathcal{H}}.$ This approach is justified
since $\sigma(P_nA\lceil_{P_n\mathcal{H}}) \rightarrow \sigma(T)$ as
$n \rightarrow \infty.$ By the assumption on the complexity index in
finite dimensions it follows that if $\mathcal{E}$ denotes the set of
compact operators then $\mathrm{SCI}(\Xi) \leq 2.$
\end{itemize}
\end{example}

\section{The Compact Case}

The reasoning in the example does not say anything about what the
Solvability Complexity Index of spectra of compact operators is, it only suggest
that the standard way of approximating spectra of such operators will
normally make use of a construction requiring two limits, and hence it
gives us an upper bound.
However, we may very well ask the question: if $\mathcal{E}$ is the set of
compact operators on $\mathcal{H}$ and $\Xi(T) = \sigma(T)$ for $T \in
\mathcal{E},$ what is $\mathrm{SCI}(\Xi)$? This is the topic of the
next theorem.

\begin{theorem}\label{compact}
Let $\{e_j\}_{j \in \mathbb{N}}$ be a basis for the Hilbert 
space $\mathcal{H}$. Let 
$$
\mathcal{E} = \{T \in \mathcal{B}(\mathcal{H}): T \, \text{is compact}\}
$$
and
$\Xi : \mathcal{E} \rightarrow \mathcal{F}$ be defined by $\Xi(T) = \sigma(T).$ Then
$\mathrm{SCI}(\Xi) = 1.$ 
\end{theorem}
\begin{proof}
Let, for $n \in \mathbb{N}$, $P_n$ denote the projection onto 
$\mathrm{span}\{e_1,\hdots,e_n\}.$ 
Define also
\begin{equation}\label{theta_k}
\Theta_n = \{z \in \mathbb{C}: 
\Re z, \Im z = r\delta, r \in \mathbb{Z}, |r| \leq n \}, \qquad 
\delta = \sqrt{\frac{1}{n}},
\end{equation}
and let 
$$
\Gamma_{n}(\{x_{ij}\}) = \{z \in \Theta_{n}
: \nexists \,
L \in LT_{\mathrm{pos}}(P_n\mathcal{H}), T_{1/n}(z) = LL^*\},
$$
where $LT_{\mathrm{pos}}(P_n\mathcal{H})$ denotes the lower
triangular matrices in $P_n\mathcal{H}$ (with respect to $\{e_1,\hdots,e_n\}$)  with positive
elements on the diagonal and 
$$
T_{1/n}(z) = (P_n(T - z)P_n)^*P_n(T - z)P_n - \frac{1}{n^2}P_n.
$$ 
We claim that $\Gamma_{n}$ can be evaluated using only
finitely many arithmetic operations and radicals of the matrix
elements $\{x_{ij}\}.$ Indeed, to determine if $z \in \Theta_k$ is in 
$\Gamma_{k}(\{x_{ij}\})$ one has to determine if $ T_{1/n}(z)$ has a
Cholesky decomposition or not. This can indeed be done by using
finitely many arithmetic operations and radicals of $\{x_{ij}\}.$ The
fact that $\Theta_n$ contains only finitely many elements yields the
assertion.
To see that  
$$
d_H(\Gamma_{n}(\{x_{ij}\}),\sigma(T)) \longrightarrow 0, \qquad n
\rightarrow \infty,
$$
one uses that fact that $\sigma(P_nKP_n)\rightarrow \mathrm{sp}(T)$ and the definition of the pseudospectrum. We omit the details. 
\end{proof}

This gives us the obvious corollary for operators on finite
dimensional Hilbert spaces.

\begin{corollary}
Let $\{e_j\}_{j =1}^N$ be a basis for the Hilbert 
space $\mathcal{H}$, and suppose that $ 5 \leq N < \infty$. Let 
$
\mathcal{E} = \mathcal{B}(\mathcal{H})
$
and
$\Xi : \mathcal{E} \rightarrow \mathcal{F}$ be defined by $\Xi(T) = \sigma(T).$ Then
$\mathrm{SCI}(\Xi) = 1.$ 
\end{corollary}

\section{The Bounded Case}\label{compute}
As we saw in the previous section, handling the compact case
essentially relies on the fact that the spectral properties of a
section $P_nTP_n$ of a compact operator $T$ resemble the spectral
properties of $T$ for large $n$ (here $P_n$ is as in the proof of
Theorem \ref{compact}). This may not be the case for an arbitrary
bounded operator. In this case completely different techniques must be
used.

\begin{theorem}\label{the_bounded} \cite{Hansen2} 
Let $\{e_j\}_{j \in \mathbb{N}}$ be a basis for the Hilbert 
space $\mathcal{H}$ and let 
$
\mathcal{E} = \mathcal{B}(\mathcal{H}).
$
Define,
for $n \in \mathbb{Z}_+, \epsilon > 0,$ the set valued functions 
$$
\Xi_1, \Xi_2, \Xi_3,\Xi_4, \Xi_5: \mathcal{E} \rightarrow \mathcal{F}
$$ by
$$ 
\Xi_1(T) = \overline{\sigma_{n,\epsilon}(T)}, \quad \Xi_2(T) =\overline{
\omega_{\epsilon}(\sigma(T))}, \quad \Xi_3(T) = \sigma(T),
$$ 
$$
\Xi_4(T) = \overline{\sigma_{\mathrm{res},\epsilon}(T)}, \quad  \Xi_5(T) = \overline{\sigma_{\mathrm{res*},\epsilon}(T)}.
$$
Then 
$$
\mathrm{SCI}(\Xi_1) \leq 2, \qquad \mathrm{SCI}(\Xi_2) \leq 3,
\qquad \mathrm{SCI}(\Xi_3) \leq 3,
$$
$$
\mathrm{SCI}(\Xi_4) \leq 2, \qquad \mathrm{SCI}(\Xi_5) \leq 2.
$$
\end{theorem}

 We will not prove the theorem here, but refer to \cite{Hansen2} for
 details. However, we will sketch the ideas on how to 
 construct the set of estimating functions. 
We start by constructing a set of estimating
functions for $\Xi_1.$
In particular, given the
 matrix elements  
 $x_{ij} = \langle Te_j, e_i \rangle$ of $T \in
 \mathcal{B}(\mathcal{H})$ we define the following:  
\begin{equation}\label{Gamma}
\begin{split}
\Gamma_{n_1,n_2}(\{x_{ij}\}) & = \{z \in \Theta_{n_1}
: \nexists \,
L \in LT_{\mathrm{pos}}(P_{n_1} \mathcal{H}),
T_{\epsilon, n_1, n_2}(z)
= LL^*\}
\\& \qquad \quad \cup \{z \in \Theta_{n_1}
: \nexists \,
L \in LT_{\mathrm{pos}}(P_{n_1} \mathcal{H}), \widetilde
T_{\epsilon, n_1, n_2}(z)
= LL^*\},\\
\Gamma_{n_1}(\{x_{ij}\}) &= \{z \in \Theta_{n_1}
: (-\infty, 0] \cap \sigma(T_{\epsilon, n_1}(z)) \neq \emptyset\} \\
& \qquad \quad \cup \{z \in \Theta_{n_1}
: (-\infty, 0] \cap \sigma(\widetilde T_{\epsilon,n_1}(z))\neq \emptyset\}, \qquad n_1, n_2 \in \mathbb{N},
\end{split}
\end{equation} 
where 
$LT_{\mathrm{pos}}(P_m \mathcal{H})$ denotes the set of lower
  triangular matrices in $\mathcal{B}(P_m \mathcal{H})$ (with respect to $\{e_j\}$)
  with strictly positive diagonal elements and
\begin{equation}\label{allT}
\begin{split}
T_{\epsilon, n_1, n_2}(z) &= T_{n_1, n_2}(z) - \epsilon^{2^{n+1}} I,\\
\widetilde T_{\epsilon, n_1, n_2}(z) &= \widetilde T_{n_1, n_2}(z) -
\epsilon^{2^{n+1}} I, \\
T_{\epsilon, n_1}(z) &= T_{n_1}(z) - \epsilon^{2^{n+1}} I,\\
\widetilde T_{\epsilon, n_1}(z) &= \widetilde T_{n_1}(z) - \epsilon^{2^{n+1}} I,
\end{split}
\end{equation}  
where $T_{n_1, n_2},$ $\widetilde T_{n_1, n_2},$
$T_{n_1}$ and $\widetilde T_{n_1}$ are defined by
\begin{equation*}
\begin{split}
T_m(z) &= P_m((T -
  z)^*)^{2^n} (T - 
  z)^{2^n}\vert_{P_m\mathcal{H}},\\  
T_{m,k}(z) &=  P_m((P_k(T - z)P_k)^*)^{2^n} (P_k(T -
  z)P_k)^{2^n}\vert_{P_m\mathcal{H}},\\
\widetilde T_m(z) &= P_m(T -
  z)^{2^n} ((T - z)^*)^{2^n}\vert_{P_m\mathcal{H}},\\
\widetilde T_{m,k}(z) &= P_m(P_k(T - z)P_k)^{2^n} ((P_k(T -
  z)P_k)^*)^{2^n}\vert_{P_m\mathcal{H}}.
\end{split}
\end{equation*}
One can then argue that the evaluation of
$\Gamma_{n_1,n_2}(\{x_{ij}\})$ indeed requires only finitely many
arithmetic operations and radicals. This is pretty straightforward,
however, showing that  
\begin{equation*}
\begin{split}
\Xi_1(T) &= \lim_{n_1 \rightarrow \infty} \Gamma_{n_1}(\{x_{ij}\}), \\
\Gamma_{n_1}(\{x_{ij}\}) &=
  \lim_{n_2 \rightarrow \infty} \Gamma_{n_1, n_2}(\{x_{ij}\}),\\
\end{split}
\end{equation*}
requires some more work. As long as we can establish that our
construction actually yields a set of estimating functions of order
two for $\Xi_1$ we can deduce that  
$\mathrm{SCI}(\Xi_1) \leq 2$ and then Theorem \ref{bounded}
yields that $\mathrm{SCI}(\Xi_2) \leq 3$. The fact that 
$\mathrm{SCI}(\Xi_3) \leq 3$
is clear from the fact that $\mathrm{SCI}(\Xi_1) \leq 2$.

\begin{remark}
It is tempting to try clever ways of subsequencing in order to reduce
the bound on the Solvability Complexity Index. However, the trained
eye of an operator theorist will immediately spot the difficulties
with such a strategy. This is confirmed in the following proposition.
\end{remark}

\begin{proposition}
Let $\epsilon > 0$, $n = 0$ and $\Gamma_{n_1,n_2}$ be defined as in
(\ref{Gamma}). There does NOT exists a subsequence
$\{k_m\}_{m\in\mathbb{N}}$ such that  
$$
\Gamma_{m,k_m}(\{x_{ij}\}) \longrightarrow \overline{\sigma_{\epsilon}(T)},
\qquad m \rightarrow \infty, \quad x_{ij} = \langle Te_j, e_i \rangle, \quad \forall \, T \in \mathcal{B}(\mathcal{H}).
$$
\end{proposition}
\begin{proof}
We will argue by contradiction and 
suppose that such a subsequence exists. Then, obviously may assume that $k_m > m$ for
all $m \in \mathbb{N}$. 
Now define the operator $S$ on $\mathrm{span}\{e_j\}_{j\in\mathbb{N}}$ by 
$$
\langle S e_j ,e_i\rangle = 
\begin{cases} 1+ \epsilon & i = m, j = k_m +1,
\\
0 &\text{otherwise}.
\end{cases}
$$
It is easy to see that $S$ extends to a bounded operator on $\mathcal{H}$ with $\|S\| \leq 1$. Define $T = S + S^*.$ Then a 
straightforward computation gives that for $m >1$ (since $k_m > m$ ) 
it follows that  
\begin{equation}\label{wrong}
P_mT^*TP_m =(1+\epsilon)^2 P_m, \qquad P_mT^*P_{k_m}TP_m = (1+\epsilon)^2 P_{m-1}.
\end{equation}
And it is this equation that will yield the contradiction. 
Letting $x_{ij} = \langle T e_j ,e_i\rangle$ we have (as argued above) that 
$
\Gamma_{m}(\{x_{ij}\}) \longrightarrow \overline{\sigma_{\epsilon}(T)}, $
as $m\rightarrow \infty,
$
and by hypothesis that 
$$
\Gamma_{m,k_m}(\{x_{ij}\}) \longrightarrow
\overline{\sigma_{\epsilon}(T)}, \qquad m\rightarrow \infty, 
$$
hence
\begin{equation}\label{contr}
d_H(\Gamma_{m}(\{x_{ij}\}), \Gamma_{m,k_m}(\{x_{ij}\}))
\longrightarrow 0, \qquad m \rightarrow \infty. 
\end{equation}
It is the latter statement that is not true, and therefore gives us the desired contradiction. 
We will now demonstrate why.
Note that 
\begin{equation*}
\begin{split}
\Gamma_{m}(\{x_{ij}\}) &= \{z \in \mathbb{C}
: (-\infty, 0] \cap \sigma(T_{\epsilon, m}(z)) \neq \emptyset\} \\
& \qquad \cup \{z \in \mathbb{C}
: (-\infty, 0] \cap \sigma(\widetilde T_{\epsilon,m}(z))\neq \emptyset\}, 
\end{split}
\end{equation*}
where 
$$
T_{\epsilon, m}(z) =  P_mT^*T\vert_{P_m\mathcal{H}} - P_m \bar zT\vert_{P_m\mathcal{H}} - 
P_m zT^*\vert_{P_m\mathcal{H}} - (\epsilon - |z|)I,
$$
$$
\widetilde T_{\epsilon, m}(z) =
P_mT^*T\vert_{P_m\mathcal{H}} - P_m \bar zT\vert_{P_m\mathcal{H}} - 
P_m zT^*\vert_{P_m\mathcal{H}} - (\epsilon - |z|)I.
$$
Hence, by the first part of (\ref{wrong}) and the fact that $\|T\| \leq 2$ it follows that 
$$
(-\infty, 0] \cap \sigma(T_{\epsilon, m}(z)) = \emptyset, \qquad 
(-\infty, 0] \cap \sigma(\widetilde T_{\epsilon,m}(z))=  \emptyset, \qquad |z| \leq 1/8.
$$
In particular,
\begin{equation}\label{splat}
\mathbb{D}(0,1/8) \cap \Gamma_{m}(\{x_{ij}\}) = \emptyset, \quad \forall \, m \in \mathbb{N}\setminus \{1\},
\end{equation} 
where $\mathbb{D}(0,1/8)$ denotes the closed disc centered at zero with
radius $1/8$. On the other hand we have that
\begin{equation*}
\begin{split}
\Gamma_{m,k_m}(\{x_{ij}\}) &= \{z \in \mathbb{C}
: (-\infty, 0] \cap \sigma(T_{\epsilon, m,k_m}(z)) \neq \emptyset\} \\
& \qquad \cup \{z \in \mathbb{C}
: (-\infty, 0] \cap \sigma(\widetilde T_{\epsilon,m,k_m}(z))\neq \emptyset\}, 
\end{split}
\end{equation*}
where 
$$
T_{\epsilon,m,k_m}(z) =  P_m(T - z)^*P_{k_m}(T -
  z)\vert_{P_m\mathcal{H}} - \epsilon I,
  $$
$$
\widetilde T_{\epsilon,m,k_m}(z) =  P_m(T - z)P_{k_m}(T -
  z)^*\vert_{P_m\mathcal{H}} - \epsilon I.
$$
Thus, by the last part of (\ref{wrong}) it follows that 
$$
0 \in \sigma(T_{\epsilon, m,k_m}(z)), \qquad z = 0 \quad \forall \, m \in \mathbb{N}\setminus \{1\}.
$$
In particular,
$0 \in \Gamma_{m,k_m}(\{x_{ij}\})$ for all $m > 1$. This, together with (\ref{splat}) contradicts
(\ref{contr}), and we are done.

\end{proof}

Although, as the previous proposition suggest, clever choices of sub sequences 
are not going to help, extra structure on the operators does help. This
is documented in the next theorem.

\begin{theorem} \cite{Hansen2} 
Let $\{e_j\}_{j \in \mathbb{N}}$ be a basis for the Hilbert 
space $\mathcal{H}$ and let $d$ be a positive integer. Define 
$$
\mathcal{E} = \{T \in \mathcal{B}(\mathcal{H}):\langle Te_{j+l},
e_j\rangle = \langle Te_j, e_{j+l}\rangle = 0, \quad l > d \}.
$$
Let $\epsilon > 0$ and $n \in \mathbb{Z}_+$ and 
$\Xi_1, \Xi_2, \Xi_3: \mathcal{E} \rightarrow \mathcal{F}$ be defined by
$
\Xi_1(T) = \overline{\sigma_{n,\epsilon}(T)}$, $\Xi_2(T) =\overline{
\omega_{\epsilon}(\sigma(T))}$ and $\Xi_3(T) = \sigma(T).
$ Then 
$$
\mathrm{SCI}(\Xi_1) =1, \qquad \mathrm{SCI}(\Xi_2) \leq 2,
\qquad \mathrm{SCI}(\Xi_3) \leq 2.
$$
\end{theorem}
The proof of this theorem is a little involved and we refer to \cite{Hansen2} for details, however, we will sketch 
the ideas.
To prove the theorem one defines a set of estimating function for $\Xi_1$ (containing only one element) 
as follows. For $T \in \mathcal{E}$, $x_{ij} = \langle Te_j, e_i \rangle$ and 
 $\Theta_k$ defined as in (\ref{theta_k}), with $k \in \mathbb{N},$ we let
\begin{equation*}
\begin{split}
\Gamma_k  &= \{z \in  \Theta_k
: \nexists \,
L \in LT_{\mathrm{pos}}(P_k \mathcal{H}), 
T_{\epsilon,k,2^nd + k  }(z)
= LL^*\}
\\& \qquad \quad \cup \{z \in  \Theta_k
: \nexists \,
L \in LT_{\mathrm{pos}}(P_k \mathcal{H}), \widetilde
T_{\epsilon, k,2^nd + k}(z) = LL^*\},
\end{split}
\end{equation*}
where $T_{\epsilon,k,2^nd + k  }(z)$ and $\widetilde
T_{\epsilon, k,2^nd + k}(z)$ are defined in (\ref{allT}).

\section{The Unbounded Case}
So far we have only considered bounded operators, however, in quantum
mechanics it is really the unbounded operators that have the most
impact. In particular, we are interested in determining spectra of
Schr\"{o}dinger operators
$$
H = -\Delta + V,
$$
and more specifically we are interested in non-self-adjoint Schr\"{o}dinger
operators. The reason why we need non-self-adjointness in Quantum
Mechanics is threefold:
\begin{itemize}
\item[(i)] Open Systems: If the system is open, meaning that particles
  will enter or exit (or both) on cannot have energy
  preservation. Thus the time evolution operator 
$e^{-itH}$ cannot be unitary and hence $H$ cannot be self-adjoint.
\item[(ii)] Closed Systems (PT-symmetry, alternative inner products) \cite{Bender3}: 
Physicists have recently considered the possibility that the usual inner
  product on $L^2(\mathbb{R}^d)$ can be replaced by a different inner
  product. Thus a Schr\"{o}dinger operator on $L^2(\mathbb{R}^d)$ may be
  self-adjoint with the usual inner product, however, non-self
  adjoint with another inner product, and vice versa.
\item[(iii)] Resonances \cite{Zworski2}: This is a well know phenomenon in Quantum
  Mechanics that yield non-self-adjoint operators.      
\end{itemize}

Fortunately, we have bounds on the Solvability Complexity Index also
for unbounded operators.

\begin{theorem}\label{total} \cite{Hansen2} 
Let $\{e_j\}_{j \in \mathbb{N}}$ and 
$\{\tilde e_j\}_{j \in \mathbb{N}}$ be
bases for the Hilbert 
space $\mathcal{H}$ and let 
\begin{equation*}
\begin{split}
\mathcal{\tilde E} &= \{T \in \mathcal{C}(\mathcal{H}\oplus \mathcal{H}): 
T = T_1
\oplus T_2, T_1 , T_2 \in \mathcal{C}(\mathcal{H}), T_1^* = T_2\}\\
\mathcal{E} &=\{T \in \mathcal{\tilde E}: \mathrm{span}\{e_j\}_{j \in \mathbb{N}}\, 
\text{is a core
for}\, \, T_1, \, 
\mathrm{span}\{\tilde e_j\}\, \text{is a core
for} \,\, T_2\}. 
\end{split}
\end{equation*}
Let $\epsilon > 0,$ 
$\Xi_1: \mathcal{E} \rightarrow \mathcal{F}$ and 
$\Xi_2: \mathcal{E} \rightarrow \mathcal{F}$ be defined by 
$
\Xi_1(T) = 
\overline{\sigma_{\epsilon}(T_1)}$ and $\Xi_2(T) = \sigma(T_1).
$
Then 
$$\mathrm{SCI}(\Xi_1) \leq 2, \qquad \mathrm{SCI}(\Xi_2) \leq 3.$$
\end{theorem}

\begin{corollary}
Let $\{e_j\}_{j \in \mathbb{N}}$ be a basis for the Hilbert 
space $\mathcal{H}$ and let 
$$
\mathcal{E} = \{A \in \mathcal{SA}(\mathcal{H}): 
\mathrm{span}\{e_j\}_{j \in \mathbb{N}}\, \text{is a core for}\, A\}.
$$ 
Let $\epsilon > 0$  and 
$\Xi_1, \Xi_2: \mathcal{E} \rightarrow \mathcal{F}$ be defined by 
$
\Xi_1(T) =
\sigma(T)$ and $\Xi_2(T) =
\overline{\omega_{\epsilon}(\sigma(T))}.
$ Then 
$$
\mathrm{SCI}(\Xi_1) \leq 3, \qquad \mathrm{SCI}(\Xi_2) \leq 2.
$$
\end{corollary}

\section{Polynomial numerical hulls}

Polynomial numerical hulls were defined in \cite{Nevanlinna_book}  with an extension to convergence properties of approximations in \cite{Nevanlinna_95}. Basic properties and related later work are highlighted in \cite{Davies_book}, see also \cite{Greenbaum1, Greenbaum2}.
The original definitions   for polynomial numerical hulls were for
bounded operators in Banach space but  we here carry the discussion
within a unital  Banach algebra $\mathcal A$.  
Given  an element $a\in \mathcal A$  and a monic polynomial $p$ we set
\begin{equation}\label{aivanperusvertailu}
V_p(a)= \{z\in\mathbb C \ : \ |p(z)| \le  \|p(a)\| \}.
\end{equation}
It follows from the {\it spectral mapping theorem}  that such a set is
always an inclusion set for the spectrum 
$$
\sigma(a)\subset V_p(a).
$$
The set $V_p(a)$ can have at most $d$ components where $d$ denotes the
degree of the polynomial. It follows from the maximum principle that
each component is simply connected and as a consequence possible holes
in the spectrum cannot be uncovered with this tool. On the other hand,
denoting by $\hat{\sigma}(a)$ the {\it polynomially convex hull} of
the spectrum, it is well known that
the following holds. 
\begin{proposition}\label{alialgebranspektri}
The spectrum of $a\in \mathcal A$ considered as an element in the
subalgebra it generates is  $\hat{\sigma}(a)$. 
\end{proposition}
\begin{example}
If $S \in \mathcal B(l^2(\mathbb{Z}))$ denotes the  bilateral shift, then the
spectrum is the unit circle but within the subalgebra generated by it
it is the whole closed disc. 
\end{example}
Let  us denote by $\mathbb P^{0}$ the set of all monic polynomials and by 
$ \mathbb P_d ^{0}$ the set of monic polynomials of degree at most $d$.

\begin{theorem} We have
\begin{equation}\label{olavinlause}
\bigcap_{p\in \mathbb P^0} V_p(a) =\hat{\sigma}(a).
\end{equation}
\end{theorem}
For a proof see \cite{Nevanlinna_book} or \cite{Davies_book}.  
The {\it polynomial numerical hull of order k} is the set
$$
V^k(a):= \bigcap_{p\in \mathbb P_k^0} V_p(a).
$$
If $T \in \mathcal B(\mathcal H)$ then  denote by $W(T)$ the {\it numerical range} of $T$. Recall that the numerical range may not be closed in infinite dimensional cases.
\begin{theorem}
For $T \in \mathcal B(\mathcal H)$ we have
$$
 V^1(T) = \overline {W(T)}.
$$
\end{theorem}
  The proof is in \cite{Nevanlinna_book}. The theorem has actually a simple form also  for operators in Banach spaces. For results on higher order numerical hulls  we refer to \cite{Davies_book}.

\section{Results based on polynomial numerical hull techniques}
 
The discussion here is formulated  in unital Banach algebras.  For basic properties of spectral theory we refer to \cite{Aupetit}. When we deal with  operators in Hilbert or Banach spaces spectral information often enters naturally via a black box of the form: input a vector, output the image under the operator.  So, in particular,  the boundary of the spectrum is a subset of the approximate point spectrum, which is determined exactly by operator vector multiplication. More specifically, for $T \in \mathcal{B}(\mathcal{X}),$ we have
\begin{equation}\label{approxpointspektr}
\sigma_{ap}(T) =\{ z \in \mathbb C \ : \   \inf_{\|x\|=1} \|(T-z)x\|=0 \}.
\end{equation}
This allows us to search for points in the spectrum given that we can carry out operator vector multiplication. The situation in an abstract Banach algebra is very different.  
In short: for $a\in \mathcal A$, in order to determine the spectrum, we are supposed to {\it search for} complex numbers $z$ such that the resolvent {\it does not} exist.  This is not a very useful starting point for computations.  The picture reverses itself almost automatically: it is more natural to search for $z$ such that 
$$
(z-a)^{-1}
$$
exists, 
allowing the possibility for a constructive approach.  Clearly, this happens if and only if there is a $b(z) \in \mathcal A$ such that
$$
\| (z-a) b(z) -1\| <1.
$$
In this formulation it is also  clear that  if such an element  $b(z)$ in the subalgebra generated by $a$ is constructed, we 
know that the point is not in the spectrum, and also we can write down the inverse.  A typical algorithm in this type of case tries to find such an element and  terminates after it has been found, and otherwise it runs for ever.  In this sense  deciding whether a given complex number  $z$ is  outside of $\hat{\sigma}(a)$  appears easier than deciding whether $z\in\hat{\sigma}(a)$. In some contexts people speak of {\it semi-decidable} situations.
 We start by stating a "metatheorem".
In what follows we shall  assume that the following  operations are  available.
\begin{itemize}
\item[(i)]
Given $a\in \mathcal A$ one can form powers of $a$ and combine them into polynomials $p(a)=\sum_{j=0} ^d \alpha_j a^j$with complex coefficients. 
\item[(ii)]
Given $p(a) \in \mathcal A$ we can evaluate  its norm $\|p(a)\|$.
 \item[(iii)]
 Normal computations with complex numbers are assumed, for example we can ask for given $z\in \mathbb C$ whether
 $$
 |p(z)| > \|p(a)\|.
 $$
 \item[(iv)] What we are {\it not} assuming is the operation  of inversion for invertible elements:
$$
a \mapsto a^{-1}.
$$
\end{itemize}

\begin{proposition} {\bf "Metatheorem"}\label{meta}
Suppose you could compute, with a finite number of operations of (i)-(iii),  for given $a\in \mathcal A$,  compact  sets $K_n(a) \subset \mathbb C$ (where $n \in \mathbb{N}$) such that 
\begin{itemize}
\item[(v)]
testing whether a complex number is inside $K_n(a)$ requires only a finite number of steps, 
 \item[(vi)]
$$
K_{n+1}(a) \subset K_n(a) 
$$
\item[(vii)] 
$$
\hat{\sigma}(a) = \bigcap_{n \ge 1} K_n(a). 
$$
\end{itemize} 
Then the question, whether
$$
z\in \hat{\sigma}(a),
$$
would be answered in finite number of steps if the answer is negative, and otherwise  the process would run for ever.
\end{proposition}

\begin{remark}
Observe that since the  compact sets  are in inclusion, (vii) immediately implies
the convergence in the Hausdorff metric
$$
d_H(K_n(a), \hat{\sigma}(a)) \rightarrow 0.
$$
\end{remark} 

The next result contains an actual construction of such compacts sets $K_n(a)$.

\begin{theorem} {\cite{Nevanlinna}} \label{olemassaolo}
There exists a recursive procedure which  uses only (i)-(iii) and creates (v)-(vii).   The sets  $K_n(a)$ are of the form
\begin{equation}\label{perustesti}
V_{p_n}(a) = \{ z \in \mathbb C \ : \ |p_n(z)| \le \|p_n(a)\| \},
\end{equation}
where $p_n$ are monic polynomials. 
\end{theorem}

Thus, in particular, testing whether $z$ is in or out of $V_p$ is easy.   The   heart of the process is in computing the sequence of polynomials.  For each polynomial the procedure terminates in finite number of steps, but  in addition to (i)-(iii) a "step" contains a finite number of minimization problems on norms of polynomials of fixed degree.   The construction of  the polynomials is rather complicated with quite a lot of effort needed to guarantee that the sets are in inclusion.  As a byproduct, the degrees of the polynomials grow rather rapidly.  Next we present another result, which is weaker in conclusions but much simpler in construction.

\begin{theorem}{\cite{Nevanlinna}}
Let $p_j$ be a monic polynomial such that for all monic polynomials $p$ of degree $j$ one has
\begin{equation}\label{mintest}
\|p_j(a)\| \le \|p(a)\|.  
\end{equation}
Denoting
$$
Z =\bigcap_{j\ge 1}  V_{p_j}(a)
$$
and 
$$
Z_n=\bigcap_{j=1}^{n}  V_{p_j}(a),
$$
we have 
\begin{equation}\label{inkl}
\hat\sigma(a)Ê\subset Z  
\end{equation} 
   and
\begin{equation}\label{getclose}
\sup_{z\in \partial \hat\sigma(a)}  \mathrm{dist}(z, \  \mathbb C\setminus Z_n) \longrightarrow 0, \qquad n \rightarrow \infty.
\end{equation}
\end{theorem}

Above we denote  $\mathrm{dist}( z, W) = \inf_{w\in W} |z-w|$. Observe that the conclusion (\ref{getclose}) is weaker than convergence in the Hausdorff metric
$$
d_H(\hat{\sigma}(a), Z_n) \longrightarrow 0, \qquad n \rightarrow \infty.
$$
 We can  simplify the statements above, based on  Theorem  \ref{olemassaolo}.  In fact, since there  is a constructive way to create a sequence of polynomials  $\{p_n\}$ such that    the sets  $\{V_{p_n}(a)\}$   are in inclusion and satisfy
\begin{equation}
\bigcap_{n\ge1}V_{p_n}(a) = \hat{\sigma}(a),
\end{equation}
we could in principle do the following:  enumerate all polynomials with  coefficients having rational real and imaginary parts.   Start computing the   sets $V_p(a)$ and keep in memory only the smallest set so far computed.   By Theorem \ref{olemassaolo} this trivial (but of course impractically slow) procedure creates  a converging sequence and the essential requirements are:
 
\begin{itemize}
\item[(viii)]  
one is given the enumeration of  monic  polynomials with rational coefficients,
\item[(ix)]  one can test whether $V_{p}(a) \subset V_{q}(a)$ for any
  two polynomials $p$ and $q$.

\end{itemize}
Observe that  since the boundaries of $V_p(a)$ are given by lemniscates (ix) reduces to checking whether given lemniscates   intersect and if not, each one has only finitely many components and one has to check which one is  subset of what.  Clearly all this is relatively easy to do.   We summarize:

\begin{algorithm}[{\bf based on enumeration}]
Assume (i)-(iii) and (viii)-(ix).
\begin{itemize}
\item Find an enumeration $\{p_n\}_{n\in \mathbb{N}}$
  of all monic polynomials with rational coefficients.
\item Compute $V_{p_1}(a)$ and set $K_1(a)=V_{p_1}(a)$.
\item When $K_m(a)$ has been defined with $K_m(a)=V_{p_{n_m}}(a)$  keep computing $V_{p_n}(a)$ for $n=n_m+1, \dots$ and define $K_{m+1}(a)$ when you first time meet a subset of $K_m(a)$.
\end{itemize}
\end{algorithm}

\begin{theorem}
The Algorithm based on enumeration  satisfies  the assumptions of  Proposition \ref{meta}.   
\end{theorem}

\begin{remark}\label{haave}
If the norm of any element in the algebra can be computed with a countable number of elementary steps, then  if the solvability complexity index would be defined for this type of general algorithm, the index would be at most 2.
\end{remark}

\section{Representing the resolvent}
We shall now assume that  we have a  monic polynomial $p$, which could come from one of the procedures  discussed in the previous section.
Suppose
$$
z \notin V_p(a)
$$
where $V_p(a)$ is given by (\ref{perustesti}).  
Suppose  $p\in \mathbb P_d$ is of the form
\begin{equation}\label{poly}
p(z)=z^d+ \alpha_1 z^{d-1}+ \cdots+ \alpha_d.
\end{equation}
We  put for $j=0, 1, \dots, d-1$ 
\begin{equation}\label{auxpol}
Q_j(z) =z^j+ \alpha_1 \lambda^{j-1}+ \cdots+ \alpha_j
\end{equation}
and then, with $a\in \mathcal {A},$
\begin{equation}\label{kuu}
q(z,a)=\sum_{j=0}^{d-1} Q_{d-1-j}(z) a^j.
\end{equation}
One checks easily that then
\begin{equation}\label{faktori}
(a-z)q(z,a) = p(a)-p(z).
\end{equation}
Since we assume (\ref{perustesti}),  $p(a)-p(z)$ has an inverse in the form of a convergent power series and  we obtain from (\ref{faktori}) a representation for  the resolvent in an explicit form:

\begin{proposition}
Let $a \in \mathcal{A}$ and $z \notin V_p(a)$, then 
\begin{equation}\label{repr}
(z-a)^{-1} =  \frac{q(z,a)}{p(z)} \sum_{j=0}^\infty  \frac{p(a)^j}{p(z)^j}.
\end{equation}
\end{proposition}
What is remarkable here is that this {\it single} representation
works everywhere outside a computed set $V_p(a)$ and one does {\it
  not} have to know the spectrum of $a\in \mathcal A$ in order  to
apply it.   
\begin{example}
Suppose you want to  compute $\log (a)$.  The approach requires the
following. You first run  some version of the algorithms of producing  polynomials, e.g. the one in Theorem \ref{olemassaolo}, and you keep checking whether  
\begin{equation}\label{onkologolemassa}
0\notin V_p(a).
\end{equation}
If this never happens, it means that either $0\in \sigma(a)$ or  the
spectrum separates 0 from $\infty$, and in these cases $\log(a)$ cannot
be consistently defined.   If, on the other hand,
(\ref{onkologolemassa}) holds for some $p$ then, as $V_p(a)$ is
polynomially convex, one could cut the plane from $0$  to $\infty$
outside of $V_p(a)$ and thus have a single valued logarithm, say
$\mathrm{Log}(z)$. Let $\gamma$ be a contour around $V_p(a)$ such that $\mathrm{Log}(z)$
is well defined along it. 
Then one can simply use the Cauchy integral formula to obtain
\begin{equation}\label{logsarja}
\log(a) = \sum_{j=0} ^\infty c_j(a) p(a)^j
\end{equation}
where
\begin{equation}
c_j(a) = \frac{1}{2\pi i}  \int_\gamma \mathrm{Log}(z) \frac{q(z,a)}{p(z)^{j+1}}dz.
\end{equation} 
\end{example}
Observe that  the representation of the resolvent gives rational approximations if  we truncate the series expansion. Thus the  approximations are analytic and  the values of the integrals   are path independent.  This is in contrast to  more standard approximations of the resolvent which are often created around selected  discrete points along $\gamma$.  Then the local approximations are typically only piecewise analytic, at best, and one  continues to compute the integrals using suitable quadrature formulas.  Here the integrals can be evaluated using the residue calculus {\it at the zeros of the polynomial} $p$. We shall now shortly  discuss the resulting algorithmic approach for  holomorphic functional calculus.

\section{Multicentric decompositions and holomorphic calculus}
Let $f$ be analytic in a domain $\Omega \subset \mathbb C$.  We assume the following on $f$ and $\Omega$.  Given a monic polynomial $p$  with roots $z_j \in \Omega$ and an element $a$ in a Banach algebra
\begin{itemize}
\item[(x)]
we can test whether $V_p(a) \subset \Omega$,
\item[(xi)]
we can evaluate the derivatives of all orders of $f$ at  the roots $z_j$.
\end{itemize}
 The algorithmic approach   starts by   searching  for a polynomial $p$ satisfying
\begin{equation}
V_p(a) \subset \Omega.
\end{equation}\label{vaatimus}  
By Theorem \ref{olemassaolo}  such a polynomial $p_n$ is found after a finitely many steps if and only if 
\begin{equation}\label{inkluusio}
\hat{\sigma}(a)\subset \Omega.
\end{equation}
In fact,  if  $U$  is open such that
$$
\bigcap_{j\ge 1} K_j \subset U,
$$
where the compact sets $K_j$ are in inclusion, then there exists also an integer $n$ such that
$$
K_n \subset U.
$$
Assume then that  a monic $p$ of degree $d$ has been found, satisfying (\ref{vaatimus}). We may further assume that $p$ has only simple zeros. Denote by $\delta_k(z)$ the interpolating polynomial of degree $d-1$:
$$
\delta_k(z) =\frac{p(z)}{p'(z_k)(z-z_k)}.
$$
One checks easily that $q$ in (\ref{kuu}) can be written as
\begin{equation}
q(z,a) = \sum_{k=1}^d p'(z_k) \delta_k(a) \delta_k(z).
\end{equation}
We could  represent $f$ in a {\it multicentric power series} or {\it Jacobi series}
as in (\ref{logsarja})
\begin{equation}\label{jacobisarja}
f(z)=\sum_{j=0}^\infty   c_j(z) p(z)^j,
\end{equation}
where
the coefficients $c_j$ are polynomials of at most degree $d-1$.  However,
we prefer  to proceed to the {\it multicentric decomposition} of $f$.  Given $f$ there are unique analytic functions $f_k$ such that
\begin{equation}\label{multidecompo}
f(z)= \sum_{k=1}^d \delta_k(z) f_k(p(z)).
\end{equation}

\begin{proposition}
If $\gamma$ is a contour surrounding all zeros $z_k$ and staying inside $\Omega$ then $f\in H(\Omega)$ has a unique multicentric decomposition (\ref{multidecompo}), where
each function $f_k$ has a power series  
$$
f_k(w)=\sum_{j=0}^\infty \alpha_{kj} w^j
$$ with
$$
\alpha_{kj} = \frac{1}{2\pi i}Ê\int_\gamma \frac{f(z)}{z-z_k}\frac{dz}{p(z)^j}.
$$
\end{proposition}

 Since $f_k$ has an explicit Taylor series at the origin, it is easy to come up with effective estimates for the coefficients  allowing one to control the accuracy in truncating the series.  
 We refer to \cite{Nevanlinna_10}  for details, but remark at the end that the coefficients can be computed recursively. Also, \cite{Nevanlinna_10}  contains references to the history of these expansions.  It is clear that we could also write the expressions in the form
\begin{equation}\label{jacobiexp}
f(z)= \sum_{j=0} ^\infty c_j(z) p(z)^j
\end{equation}
where $c_j$ are polynomials of degree at most $d-1$. Some authors call (\ref{jacobiexp}) Jacobi series and, 
in fact, C. G. J. Jacobi  \cite{Jacobi}
studied such expansions, without  the use of the Cauchy integral.

\begin{proposition}  \cite{Nevanlinna_10} Based on $p$ one can pre compute coefficients such that the following holds. 
Set first $\alpha_{k0}= f(z_k)$.  Then 
each $\alpha_{kj}$ can be computed by explicit substitution requiring the following input: $f^{j}(z_k),$ and already computed coefficients $\alpha_{m,i}$ for $m=1,\dots,d$, $i=0, \dots, j-1$.

\end{proposition}
Thus,  after the polynomial $p$ has been fixed  one obtains a countable expression for $f(a)$ such that each of the truncated expressions can be computed with a finite number of operations.  

\begin{remark}  
The multicentric  calculus above can be applied for bounded operators $T\in B(\mathcal X)$  to compute $f(T)x$ for holomorphic $f$ and a vector $x\in \mathcal X$  in such a way that  one has to create only vectors $T^j x$ {\it provided} one can  be sure that 
$$
V_p(T)  \subset \Omega.
$$
Notice that this requires   testing  $|p(z)|$ against $\|p(T)\|$, whereas testing against $\|p(T)x\|$ does not suffice. 
 \end{remark}

 \section{Open Problems}

The great open problem is to get a complete classification theory of the Solvability Complexity Index. 
However, before we can get that we must solve the following problem: Let 
$
\mathcal{E} = \mathcal{B}(\mathcal{H})
$
and $\Xi:\mathcal{E} \rightarrow \mathcal{F}$ 
be defined by $\Xi(T) = \sigma(T)$. Is
$$
\mathrm{SCI}(\Xi) > 1?
$$
We will not speculate on what kinds of techniques one should use to answer this question, but rather 
engage in some philosophical thinking about the problem. What if  $\mathrm{SCI}(\Xi) = 1$?
That means that there exists an (as of today unknown) construction that allows us to recover the spectrum 
of any operator by using only arithmetic operations and radicals and then taking {\it one} limit. Such a result would indeed be spectacular, however, a little too good to be true.

What if  $\mathrm{SCI}(\Xi) > 1$? If that is the case, it means that the Solvability Complexity Index makes 
a jump going from finite to infinite dimensions, similar to the jump between dimensions four and five.
But where does the jump occur? We have already shown that, at least, the jump cannot occur between 
finite rank operators and compact operators. Also, could it be that   $\mathrm{SCI}(\Xi) > 2$? In that case
Theorem \ref{bounded} automatically gives us that $\mathrm{SCI}(\Xi) = 3$. However, in this case there will be not only one jump, but two, and where will they occur?

We hope to report on these issues in the future, however, note that to show that the Solvability Complexity Index of the spectrum makes a jump between dimensions four and five is equivalent to showing the unsolvability of the quintic using radicals. This was done by using tools outside of analysis and we will not exclude the possibility that deep mathematics from other fields will yield the solution to the problem.
 
 The results above which are based on the  polynomial numerical hull techniques are written in  a way which does not automatically relate to the solvability complexity index as there are operations which are defined on higher level of abstraction.  However, as pointed out in  Remark \ref{haave}, if we would have effective ways to control the computation of 
$$
T \mapsto \|p(T)\|
$$
then we  would have a different  way of addressing these questions.

 \section{Addendum}
 Up to this section the paper is unchanged, modulo small corrections, from its original form that was put online in 2011. Since the conference {\it Operator Theory and Applications: Perspectives and Challenges} (2010), where some of these ideas were presented, and that was the starting point of this paper, there have been substantial developments in the field. For example, in \cite{Seidel} the techniques for computing spectra and $(n,\epsilon)$-pseudospectra discussed in Section \ref{compute} were extended to Banach space operators. Moreover, in \cite{Olavi_JFA} the multicentric calculus was used as a basis for  holomorphic functional calculus, allowing a nontrivial extension of the von Neumann theorem that for polynomials  $p$  and contractions $A$ in a Hilbert space  $\mathcal{H}$  we  have
$$
\| p(A) \| \leq  \sup_ {|z| \leq 1}  |p(z)|.
$$
Furthermore, in \cite{Olavi_Oper}, to get a functional calculus for nonholomorphic functions working for all square matrices,  a Banach algebra was constructed starting from continuous functions 
$$
f :  M \subset \mathbb C \rightarrow  \mathbb C^d
$$
and formulating a product $\circledcirc$  such that the multicentric representation   of a scalar function $\varphi$ could be viewed as the Gelfand transformation  of $f$ in that algebra. That then led  to a natural functional calculus, where for example,  differentiability of the function $\varphi$ is not necessary at  eigenvalues corresponding to nontrivial Jordan blocks.

Finally, in \cite{SCI, SCI2}, and many of the open problems have been solved. For example, expressed informally, we have the following:
 \begin{itemize}
 \item[(i)] The SCI of computing spectra of operators in $\mathcal{B}(l^2(\mathbb{N}))$ is equal to $3$.
  \item[(ii)] The SCI of computing spectra of self-adjoint and normal operators in $\mathcal{B}(l^2(\mathbb{N}))$ are  both equal to $2$.
   \item[(iii)] The SCI of computing spectra of banded normal operators in $\mathcal{B}(l^2(\mathbb{N}))$ is equal to $1$, however computing the essential spectrum has SCI equal to $2$.
 \end{itemize}
 Moreover, these results are true also when one removes the possibility of using radicals.  In addition, the concept of the SCI can be generalized to arbitrary computational problems. 
 
\subsection{The SCI hierarchy and Smale's program}
In the 1980's S. Smale initiated  \cite{Smale2, Shub_Smale_JAMS, Smale_Acta_Numerica} a comprehensive program for establishing the foundations of computational mathematics. In particular, he started by asking basic questions on the existence of algorithms. One of the fundamental algorithms in numerical analysis is Newtons method. However, the problem is that it may not converge, an issue that causes trouble in for example polynomial root finding. 
Thus, Smale \cite{Smale2} questioned whether there exists an alternative to Newton's method, namely, a purely iterative generally convergent algorithm (see \cite{Smale2} for definition). More precisely he asked: ``\emph{Is there any
purely iterative generally convergent algorithm for polynomial zero
finding ?}"   The problem was settled by C. McMullen in \cite{McMullen1} as follows: yes, if the degree is three; no, if the degree is higher (see also \cite{McMullen2}). However, in \cite{Doyle_McMullen} P. Doyle and C. McMullen demonstrated a fascinating phenomenon: this problem can be solved in the case of the quartic and the quintic using several limits. In particular, they provide a construction such that, by using several rational maps and independent limits, a root of the polynomial can be computed. In other words, this is an example of a problem where the SCI is greater than one, but still finite. 

It turns out that this phenomenon happens everywhere in scientific computing and Smale's questions and McMullens answers are indeed an example of classification problems in the SCI hierarchy that can be described as follows. Given a definition of what an algorithm can do, we have that
\begin{itemize}
 \setlength\itemsep{0em}
\item[(i)] $\Delta_0$ is the set of problems that can be computed in finite time. 
\item[(ii)] $\Delta_1$ is the set of problems with SCI $\leq 1$, where one also has error control., i.e. the algorithm halts on an input parameter $\epsilon$ and the output is no further away than $\epsilon$ from the true solution.  
\item[(iii)] $\Delta_2$ is the set of problems where SCI $\leq 1$, but error control may not be possible. 
\item[(iv)] $\Delta_{m+1}$ is the set of problems where the SCI $\leq m$.
\end{itemize}
This hierarchy shows up in many basic computational problems, and below follow some examples.
\begin{example}[Spectral problems]
Assuming that an algorithm can do arithmetic operations and comparisons of real numbers (complex numbers are treated as a pair of real numbers), then the general computational spectral problem is in $\Delta_4$, but not in $\Delta_3$. The self-adjoint spectral problem is in $\Delta_3$ but not in $\Delta_2$. The compact spectral problem is in $\Delta_2$ but not in $\Delta_1$. The finite-dimensional spectral problem is in $\Delta_1$, but not in $\Delta_0$. 
\end{example}

\begin{example}[Inverse Problems]
Assuming the same type of algorithm as in the previous example and
suppose that 
$b \in l^2(\mathbb{N})$ and $A \in \mathcal{B}_{\mathrm{inv}}(l^2(\mathbb{N}))$ (the set of bounded invertible operators), then solving 
$$
Ax = b
$$
is in $\Delta_3$ and not in $\Delta_2$. If $A$ is self-adjoint, then the problem is still in $\Delta_3$ and not in $\Delta_2$. However, if $A$ is banded then the problem is in $\Delta_2$ but not in $\Delta_1$. Finally, if $A$ is finite-dimensional, then the problem is in $\Delta_0$.
\end{example}

\begin{example}[Optimization]
Assuming the same as above, we may consider a popular problem occurring in infinite-dimensional compressed sensing \cite{BAACHGSCS, AHPRBreaking}. In particular,
given an $A \in \mathcal{B}(l^2(\mathbb{N}))$ and $y \in l^2(\mathbb{N})$ that are feasible, then the optimization problem of finding
\begin{equation*}
x \in \mathop{\mathrm{arg min}}_{\eta \in l^p(\mathbb{N}) } \| \eta \|_{l^1}\ \mbox{ subject to $\|A\eta - y \| \leq \delta$}, \qquad \delta > 0
\end{equation*}
is not in $\Delta_2$, however, it is not known where in the SCI hierarchy this problem is. 
\end{example}

An important result in \cite{SCI, SCI2} is that the SCI hierarchy does not collapse regardless of the axioms on the algorithm. This demonstrates that any theory aiming at establishing the foundations of computational mathematics will have to include the SCI hierarchy.

\bibliography{bib_file_Banach1}
\bibliographystyle{abbrv}

\end{document}